\documentclass[12pt,a4paper]{amsart}
\usepackage[latin1]{inputenc}
\usepackage{latexsym}
\usepackage{color,graphicx,shortvrb}
\usepackage{amsmath, amssymb}
\usepackage{amsfonts}
\usepackage[colorlinks, bookmarks=true]{hyperref}

\newtheorem{theorem}{Theorem}[section]

\newtheorem{proposition}[theorem]{Proposition}
\newtheorem{corollary}[theorem]{Corollary}

\theoremstyle{definition}

\newtheorem{remark}[theorem]{Remark}

\newcommand{\N}{\mathbb{N}}

\title{A simple observation about compactness and fast decay of Fourier coefficients}
\author{J. M. Almira}

\begin{document}

 \baselineskip=16pt

\numberwithin{equation}{section}

\maketitle \markboth{Compactness and fast decay of Fourier coefficients}{J. M. Almira}

\begin{abstract}
Let $X$ be a Banach space and suppose $Y\subseteq X$ is a Banach space compactly embedded into $X$, and $(a_k)$ is a weakly null sequence of
functionals in $X^*$. Then there exists a sequence $\{\varepsilon_n\} \searrow 0$
such that $|a_n(y)| \leq \varepsilon_n \|y\|_Y$ for every $n\in\mathbb{N}$ and every $y\in Y$. We prove this result and we use it for the study of fast decay of Fourier coefficients in $L^p(\mathbb{T})$ and frame coefficients in the Hilbert setting.
\end{abstract}

\section{Motivation}  One of the classical problems in Harmonic Analysis is the study of the relationship that exists between decay properties of the Fourier coefficients $c_n(f)=\frac{1}{2\pi}\int_0^{2\pi}f(t)e^{-int}dt$ of a $2\pi$-periodic function $f:\mathbb{T}\to\mathbb{C}$  and its membership to several function spaces. Just to mention a few well known examples, we show the following list of results:
\begin{itemize}
\item Riemann-Lebesgue Lemma sates that for $f\in L^1(\mathbb{T})$ the Fourier coefficients satisfy $\lim_{n\to \pm\infty}|c_n(f)|=0$.
\item Parseval's identity states that  $f\in L^2(\mathbb{T})$ if and only if $\{c_n(f)\}\in \ell^2(\mathbb{Z})$.
\item For $1<p\leq 2$, Hausdorff-Young's inequality states that if   $f\in L^p(\mathbb{T})$ then $\{c_n(f)\}\in \ell^q(\mathbb{Z})$, where $\frac{1}{p}+\frac{1}{q}=1$, and $\|\{c_n(f)\}\|_{\ell^q(\mathbb{Z})}\leq \|f\|_{L^p(\mathbb{T})}$.
\item If $p>2$ and $f\in L^p(\mathbb{T})$, then $f\in L^2(\mathbb{T})$ and  $\{c_n(f)\}\in \ell^2(\mathbb{Z})$.
\item De Leeuw, Kahane and Katznelson \cite[Theorem 2.1, page 278]{katznelson} proved that for any sequence $\{c_n\}\in \ell^2(\mathbb{Z})$ there exists a continuous function $f\in\mathbf{C}(\mathbb{T})$ such that $|c_n(f)|\geq c_n$ for all $n\in\mathbb{Z}$.
\end{itemize}
Moreover, other related results stand up the relationship that exists between fast decay of Fourier coefficients of a function $f$ and its smoothness properties:
\begin{itemize}
\item If $f$ is absolutely continuous and periodic then $c_n(f)=\mathbf{o}(1/n)$.
\item If $f$ is periodic, $k$-times differentiable and $f^{(k-1)}$ is absolutely continuous then $c_n(f)=\mathbf{o}(1/n^k)$.
\item If $f$ is of bounded variation on $\mathbb{T}$ then $c_n(f)\leq \frac{V_{[0,2\pi]}(f)}{2\pi |n|}$ for all $n\neq 0$.
\item If $f\in \mathbf{Lip}_{\alpha}(\mathbb{T})$ then $c_n(f)=\mathbf{O}(n^{-\alpha})$.
 \end{itemize}
In this note we prove a quite general result about functionals which implies that, for $1\leq p<\infty$, associated to any Banach space $Y$ compactly embedded into $L^p(\mathbb{T})$ there is a decreasing sequence $\{\varepsilon_n\} \searrow 0$
such that $|c_n(y)| \leq \varepsilon_n \|y\|_Y$ for all $n\in\mathbb{Z}$ and $y\in Y$.  We also prove an analogous result for frames in the Hilbert space setting.

\section{The main result}
\begin{proposition} Let $X$ be a Banach space and suppose $Y\subseteq X$ is a Banach space compactly embedded into $X$, and $(a_k)$ is a weakly null sequence of
functionals in $X^*$. Then there exists a sequence $\{\varepsilon_n\} \searrow 0$
such that $|a_n(y)| \leq \varepsilon_n \|y\|_Y$ for every $n\in\mathbb{N}$ and every $y\in Y$.
\end{proposition}
\begin{proof} It follows from Banach-Steinhaus Theorem and the hypothesis of $(a_n)$ being weakly null, that $\sup_n \|a_n\|=M<\infty$. For each $m \in \N$, find a finite set
$S_m \subset \{y \in Y : \|y\|_Y \leq 1\}$, such that, for every $y \in Y$
with $\|y\|_Y \leq 1$, there exists $z \in S_m$ with $\|y - z\|_X < 1/(2Mm)$.
Then there exists a sequence $N_1 < N_2 < \ldots$ such that
$|a_k(z)| < 1/(2m)$ for any $z \in S_m$ and $k \geq N_m$. By the triangle inequality,
$|a_k(y)| < 1/m$ for any $k \geq N_m$, and any $y$ in the unit ball of $Y$.
Now define $\varepsilon_i = 1/m$ for $N_m \leq i < N_{m+1}$.
For $i < N_1$, let $\varepsilon_i = M$. We have shown that
$|a_i(y)| < \varepsilon_i \|y\|_Y$ for any $i$.
\end{proof}
\begin{remark} Note that  existence of  $\{\varepsilon_n\} \searrow 0$
such that $|a_n(y)| \leq \varepsilon_n \|y\|_Y$ for every $n\in\mathbb{N}$ and every $y\in Y$ is equivalent to claim that
$\lim_{n\to\infty} \|(a_n)_{|Y}\|_{Y^*}=0$. \end{remark}
\begin{corollary} Let $1\leq p<\infty$ and let $Y$ be a Banach space compactly embedded into $L^p(\mathbb{T})$. Then there exists a decreasing sequence $\{\varepsilon_n\} \searrow 0$
such that $|c_n(y)| \leq \varepsilon_n \|y\|_Y$ for all $n\in\mathbb{Z}$ and $y\in Y$.
\end{corollary}
\begin{proof} H\"{o}lder's inequality implies that, for each $n\in\mathbb{Z}$, the coefficient functionals $c_n:L^p(\mathbb{T})\to\mathbb{C}$,  $c_n(f)=\frac{1}{2\pi}\int_0^{2\pi}f(t)e^{-int}dt$, are well defined and uniformly bounded with norm $\|c_n\|\leq 1$. Moreover, the results stated at the introductory section of this note show that  the sequence $\{c_n\}$ is weakly null.
\end{proof}
\begin{corollary}Let $H$ be a Hilbert space and let $\{\phi_n\}_{n=0}^{\infty}$ be a frame on $H$ with constants $A,B>0$. Then for every subspace $Y$ compactly embedded into $H$ there exists a decreasing sequence  $\{\varepsilon_n\} \searrow 0$
such that $|\langle y,\phi_n\rangle| \leq \varepsilon_n \|y\|_Y$ for every $n\in\mathbb{N}$ and $y\in Y$.
\end{corollary}
\begin{proof} By definition of frame, we have that, for all $x\in H$, $$A\|x\|_H^2\leq \sum_{n=0}^{\infty}|\langle x,\phi_n\rangle|^2\leq B\|x\|_H^2. $$
This shows, in particular, that the sequence of coefficient functionals $c_n:H\to\mathbb{C}$, $c_n(x)=<x,\phi_n>$ is uniformly bounded with norm $\|c_n\|\leq \sqrt{B}$, and weakly null. \end{proof}

\bigskip

\footnotesize{J. M. Almira

Departamento de Matem\'{a}ticas. Universidad de Ja\'{e}n.

E.P.S. Linares,  C/Alfonso X el Sabio, 28

23700 Linares (Ja\'{e}n) Spain

email: jmalmira@ujaen.es}

\end{document}